\documentclass[12pt]{article}
\usepackage{amsmath,amsthm,amssymb,color,amscd}

\def\b{\overline}

\def\vv{{\underline{v}}}

\def\nunu{\underline{\nu}}
\def\tt{{\underline{t}}}

\def\1{\underline{1}}
\def\0{\underline{0}}
\def\R{\mathbb R}
\def\P{\mathbb P}

\def\LL{{\cal L}}

\def\Z{\mathbb Z}

\def\C{\mathbb C}
\def\K{\mathbb K}

\def\OO{{\mathcal O}}
\def\EE{{\mathcal E}}

\newtheorem{theorem}{Theorem}

\newtheorem{lemma}{Lemma}

\theoremstyle{definition}

\newtheorem{remark}{Remark}

\title{Real Poincar\'e series of a plane divisorial valuation
}

\author{
A.~Campillo,
\and F.~Delgado \thanks{The
first two authors partially supported by
grant PID2022-138906NB-C21 funded by
MICIU/AEI/ 10.13039/501100011033 and by ERDF/EU.}\; ,
\and S.M.~Gusein-Zade\thanks{
The work of the third author (Sections~\ref{sec:Introduction}
and~\ref{sec:classical})
was supported by the grant 24-11-00124 of the Russian Science Foundation.
} }

\date{}
\begin{document}

\def\eps{\varepsilon}

\maketitle

\begin{abstract}
Earlier, there was computed the Poincar\'e series of a valuation or of a collection of valuations on
the ring of germs of holomorphic functions in two variables.
For a collection of several plane curve valuations it appeared to coincide
with the Alexander polynomial of the corresponding algebraic link.
Recently, the authors defined Poincar\'e series of a valuation or of a collection of valuations in the real setting. (Actually, there were defined three versions of them, however, one of them was found to be ``not computable''.) These two Poincar'e series were computed for one plane curve valuation. Here we compute them for a plane divisorial valuation.
\end{abstract}

\section{Introduction}\label{sec:Introduction}
Let $\K$ be either the field of complex numbers $\C$ or the field of real numbers $\R$,
let $(V,0)$ be a germ of an analytic variety over $\K$, and let $\EE^{\K}_{V,0}$ be the
ring of germs of functions on $(V,0)$.
(For $\K=\C$ (the ring of complex numbers) $\EE^{\K}_{V,0}$
is usually denoted by $\OO_{V,0}$.)
A function
$\nu:\EE^{\K}_{V,0}\to\Z_{\ge 0}\cup{\{+\infty\}}$
is a (discrete, rank one) {\em valuation} on $\EE^{\K}_{V,0}$ if
\begin{enumerate}
\item[1)] $\nu(0)=+\infty$;
\item[2)] $\nu(\lambda f)=\nu(f)$ for $f\in \EE^{\K}_{V,0}$,
$\lambda\in\K^*:=\K\setminus\{0\}$;
\item[3)] $\nu(f_1+f_2)\ge\min\{\nu(f_1),\nu(f_2)\}$ for $f_1,f_2\in \EE^{\K}_{V,0}$;
\item[4)] $\nu(f_1f_2)=\nu(f_1)+\nu(f_2)$ for $f_1,f_2\in \EE^{\K}_{V,0}$\,.
\end{enumerate}
(Sometimes, for a valuation, one demands that $\nu(f)\ne+\infty$ for $f\ne 0$. In this case a function described above is usually called a pre-valuation.)

For a valuation $\nu$ on $\EE^{\K}_{V,0}$ and $v\in\Z$, let
$J(v)=J_{\K}(v):=\{f\in \EE^{\K}_{V,0}:\nu(f)\ge v\}$.
The {\em Poincar\'e series} of the valuation $\nu$ is
\begin{equation}\label{eqn:Poincare_series}
 P_{\nu}(t)=\sum_{v=0}^{\infty}\dim_{\K} \left(J(v)/J(v+1)\right)\cdot t^v\,.
\end{equation}
For a (finite) collection of valuations the notion of
the Poincar\'e series was introduced in~\cite{CDK}
(see also below).

Let $(C,0)$ be a germ of an irreducible curve on
$(\C^2,0)$ defined by a parametrization (an uniformization)
$\varphi:(\C,0)\to(\C^2,0)$. For $f\in\OO_{\C^2,0}$, let
$\nu(f)$ be the degree of the leading term in the power
series decomposition of the function
$f\circ\varphi:(\C,0)\to(\C,0)$:
$$
f(\varphi(\tau))=a\cdot \tau^{\nu(f)} + \text{terms of higher degree}\,,
$$
where $a\ne 0$. (If $f(\varphi(\tau))\equiv 0$,
one assumes $\nu(f)$ to be equal to $+\infty$.)
The function $\nu$ is a valuation on $\OO_{\C^2,0}$
called {\em a (plane) curve valuation}.
Let $\pi:(X,D)\to(\C^2,0)$ be a modification of $(\C^2,0)$ (by a finite number of blow-ups), $D=\pi^{-1}(0)$ is
the exceptional divisor (consisting of several components
$E_{\sigma}$,
each of them being isomorphic to the complex projective line
$\C\P^1$). Let $E_{\delta}$ be  one of these components.
For $f\in\OO_{\C^2,0}$, let $\nu(f)$ be the
multiplicity of the lifting $f\circ\pi$ of the germ $f$ to the surface $X$ of the modification along the component
$E_{\delta}$. The function $\nu$ is a valuation on $\OO_{\C^2,0}$ (even not a pre-valuation in the sense described above)
called {\em a divisorial valuation}.

The Poincar\'e series of a plane curve valuation
was know to specialists in plane curve singularities
long ago. (One can see the corresponding equation, e.\,g., in~\cite{FAA-1999}.)
The Poincar\'e series of a plane divisorial valuation
was computed in~\cite{Galindo-1995}.
The Poincar\'e series of a collection of curve valuations on the ring
$\OO_{\C^2,0}$ was
defined in~\cite{CDK} and
computed in~\cite{Duke}. For a collection of divisorial valuations on this ring
it was computed in~\cite{DG}, see also~\cite{DGN-2008}.
The main (and the most effective) way to compute the
Poincar\'e series of a collection of valuations on
$\OO_{\C^2,0}$ (sometimes with additional structures, say,
with a representation of a finite group) was based on the
technique of integration with respect to the Euler characteristic over the projectivization of $\OO_{\C^2,0}$ (see, e.\,g.,~\cite{CDG-IJM-2003}). However, this method does not work (at least up to now)
in the real setting. The main problem is related with the fact that
the Euler characteristic of a complex affine space is equal to $1$,
whereas the Euler characteristic (the additive one) of a real affine space is equal
to $\pm 1$ depending on the dimension. Thus one has to use a considerably different
technique.

In \cite{BLMS-2024}, there were defined two analogues
of the Poincar\'e series in the following (real) setting.
For completeness, we shall define these Poincar\'e
series in the general case, i.\,e., for a collection
$\{\nu_i : i=1,\ldots, r\}$, of plane valuations
(that is curve or divisorial valuations on $\OO_{\C^2,0}$).
Assume that the complex plane $\C^2$ is the complexification
of a fixed real plane $\R^2\subset\C^2$. The algebra
$\EE_{\R^2,0}(=\EE^{\R}_{\R^2,0})$ is a real subalgebra
of $\OO_{\C^2,0}$.
Let us consider restrictions of the valuations $\nu_i$
to the algebra $\EE_{\R^2,0}$.
(In fact any valuation on $\EE_{\R^2,0}$ is the restriction of a valuation on
$\OO_{\C^2,0}$: see, e.\,g.,~\cite[Chapter~4, Theorem~1]{Rib}.)

For $\vv=(v_1, \ldots, v_r)\in\Z^r$, let
$$
J(\vv):=\{f\in\EE_{\R^2,0}: \nunu(f)\ge\vv \}\,.
$$
Here $\nunu(f)=(\nu_1(f),\ldots, \nu_r(f))$,
$\vv=(v_1, \ldots, v_r)\ge \vv'=(v'_1, \ldots, v'_r)$
if $v_i\ge v'_i$ for all $i$. Let
$$
\LL(\tt)=\sum_{\vv\in\Z^r}
\dim_{\R}\left(J(\vv)/J(\vv+\1)\right)\cdot\tt^{\vv},
$$
where $\tt=(t_1,\ldots, t_r)$,
$\tt^{\vv}:=t_1^{v_1}\cdots t_r^{v_r}$, $\1=(1,\ldots, 1)$.
The {\em classical Poincar\'e series} of the collection
$\{\nu_i\}$ is defined by the equation (\cite{CDK})
$$
P_{\{\nu_i\}}(\tt)=\frac{\LL(\tt)\cdot \prod_{i=1}^r (t_i-1)}{t_1\cdots t_r-1}
\in \Z[[t_1,\ldots, t_r]]\,.
$$
For $r=1$, the definition reduces to~(\ref{eqn:Poincare_series}) (with $\K=\R$).

Let
$$
F_{\vv}=\left(J(\vv)/J(\vv+\1)\right) \left\backslash
\left(\bigcup_{i=1}^{r}
J(\vv+\1_i)/J(\vv+\1)\right)\right.
\,,
$$
where
the $i$th component of
$\1_i$ is equal to $1$ and all other components
are equal to $0$.
The union
$$
\widehat{S}_{\{\nu_i\}}=\bigsqcup_{\vv\in\Z_{\ge0}^r}
F_{\vv}
$$
is called {\em the extended semigroup} of the collection $\{\nu_i\}$. (The semigroup structure
on $\widehat{S}_{\{\nu_i\}}$ is defined by the multiplication of functions;
$F_{\vv}\cdot F_{\vv'}\subset F_{\vv+\vv'}$.)
Let $\P F_{\vv}$ be the (real) projectivization of $F_{\vv}$,
i.\,e., the quotient $F_{\vv}/\R^*$.
The {\em real Poincar\'e series} of the collection
$\{\nu_i\}$ is defined by
\begin{equation}\label{eqn:real_Poincare}
P_{\{\nu_i\}}^{\chi}(\tt)=
 \sum_{\vv\in\Z_{\ge0}^r}
 \chi(\P F_{\vv})\cdot \tt^{\vv}
 \in \Z[[t_1,\ldots, t_r]]\,,
\end{equation}
where one uses {\em the additive Euler characteristic}
$\chi(\cdot)$, i.\,e., the alternating sum of the ranks of the cohomology groups with compact support. In the case of one valuation Equation~(\ref{eqn:real_Poincare}) reduces to
$$
P_{\nu}^{\chi}(t)=
 \sum_{v=0}^{\infty}
 \frac{1}{2}\cdot
 \left(1+(-1)^{\dim_{\R}\left(J(v)/J(v+1)\right)-1}\right)
 \cdot t^v.
$$
In this case, each coefficient of the power series
$P_{\nu}^{\chi}(t)$ is equal to $1$ or $0$.

\section{The minimal real resolution
of a divisorial valuation}\label{sec:resolution}

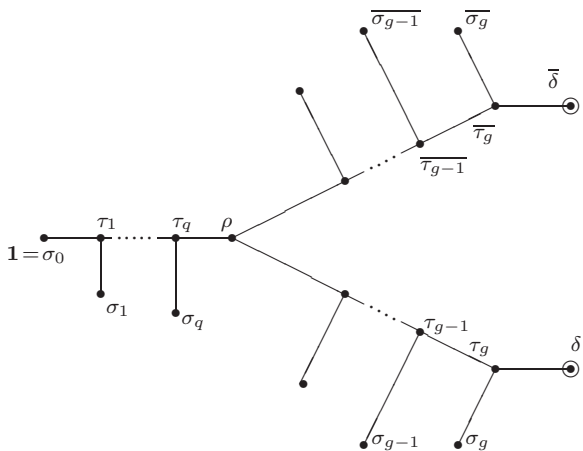
\begin{figure}[h]
$$
\unitlength=0.50mm
\begin{picture}(160.00,110.00)(0,-20)
\thinlines
\put(-10,30){\line(1,0){18}} 
\put(10,30){\circle*{0.5}}
\put(12,30){\circle*{0.5}}
\put(14,30){\circle*{0.5}}
\put(16,30){\circle*{0.5}}
\put(18,30){\circle*{0.5}}
\put(20,30){\line(1,0){20}} 
\put(40,30){\circle*{2}} 
\put(37,33){{\scriptsize$\rho$}} 

\put(40,30){\line(2,1){35}} 
\put(85,52.5){\line(2,1){25}}
\put(77,48.5){\circle*{0.5}}
\put(79,49.5){\circle*{0.5}}
\put(81,50.5){\circle*{0.5}}
\put(83,51.5){\circle*{0.5}}
\put(110,65){\circle*{2}} 
\put(70,45){\circle*{2}} 
\put(90,55){\circle*{2}} 
\put(110,65){\line(-1,2){10}} 
\put(100,85){\circle*{2}} 
\put(102,87){{\scriptsize$\overline{\sigma_{g}}$}} 
\put(70,45){\line(-1,2){12}} 
\put(58,69){\circle*{2}} 
\put(90,55){\line(-1,2){15}} 
\put(75,85){\circle*{2}} 
\put(77,87){{\scriptsize$\overline{\sigma_{g-1}}$}} 
\put(110,65){\line(1,0){20}} 
\put(130,65){\circle{4}} 
\put(130,65){\circle*{2}} 
\put(124,70){{\scriptsize$\overline{\delta}$}} 
\put(90,48){{\scriptsize$\overline{\tau_{g-1}}$}} 
\put(104,56.5){{\scriptsize$\overline{\tau_g}$}} 

\put(40,30){\line(2,-1){35}} 
\put(85,7.5){\line(2,-1){25}}
\put(77,11.5){\circle*{0.5}}
\put(79,10.5){\circle*{0.5}}
\put(81,9.5){\circle*{0.5}}
\put(83,8.5){\circle*{0.5}}
\put(110,-5){\circle*{2}} 
\put(70,15){\circle*{2}} 
\put(90,5){\circle*{2}} 
\put(110,-5){\line(-1,-2){10}} 
\put(100,-25){\circle*{2}} 
\put(103,0){{\scriptsize${{\tau_g}}$}} 
\put(91,6){{\scriptsize${\tau_{g-1}}$}} 
\put(70,15){\line(-1,-2){12}} 
\put(59,-9){\circle*{2}} 
\put(90,5){\line(-1,-2){15}} 
\put(75,-25){\circle*{2}} 
\put(110,-5){\line(1,0){20}} 
\put(130,-5){\circle{4}} 
\put(130,-5){\circle*{2}} 
\put(130,0){{\scriptsize${\delta}$}} 
\put(77,-25){{\scriptsize${\sigma_{g-1}}$}} 
\put(102,-25){{\scriptsize${\sigma_{g}}$}} 


\put(-10,30){\circle*{2}}
\put(5,30){\line(0,-1){15}}
\put(5,30){\circle*{2}}
\put(5,15){\circle*{2}}
\put(25,30){\line(0,-1){20}}
\put(25,30){\circle*{2}}
\put(25,10){\circle*{2}}

\put(-20,24){{\scriptsize ${\bf 1}\!=\!\sigma_0$}}
\put(6.5,10){{\scriptsize$\sigma_1$}}
\put(26.5,7){{\scriptsize$\sigma_q$}}
\put(4,33){{\scriptsize$\tau_1$}}
\put(24,33){{\scriptsize$\tau_q$}}
\end{picture}
$$
\caption{The minimal real resolution graph $\Gamma$ of the valuation $\nu$.}
\label{fig1}
\end{figure}

Let $\nu$ be the divisorial valuation defined by a component $E_{\delta}$ of a modification of $(\C^2,0)$.
Let $\pi:(X,D)\to(\C^2,0)$ be the minimal {\bf real} resolution of the
(divisorial) valuation $\nu$.
The dual graph $\Gamma$ of the modification $\pi$
looks like in Figure~\ref{fig1}.
(The fact that the
modification $\pi$ is real means that it is obtained by
a secuence of blow-ups such that at each step they are
made either at real points (points invariant with respect
to the complex conjugation) or at pairs of complex conjugate points simultaneously.) The vertices $\sigma$ of $\Gamma$
correspond to the components $E_{\sigma}$ of the exceptional divisor $D$ (each component is isomorphic to the complex projective line $\C\P^1$); two vertices are connected by an edge if (and only if)
the corresponding components intersect.
For $\sigma\in \Gamma$, a {\it curvette\/} at the component $E_\sigma$
is the image under the modification $\pi$
of a germ of a curve transversal to the component
$E_{\sigma}$ at a smooth point of the exceptional
divisor $D$, i.\,e., not at an intersection point with another component of $D$.

The complex conjugation acts on the surface $(X,D)$ of the resolution and also on the
resolution graph $\Gamma$.
The vertices $\sigma_i$ and $\overline{\sigma_i}$,
$i=0,1,\ldots, g$, are the {\em dead ends} of the graph $\Gamma$
($g$ is the number of Puiseux pairs of a curvette at the
component $E_{\delta}$; one may have
$\sigma_i=\overline{\sigma_i}$ for some $i$);
the vertices $\tau_i$ and $\overline{\tau_i}$,
$i=1,\ldots, g$, are the {\em rupture points} of $\Gamma$;
the vertex $\rho$ is the splitting point between the usual resolutions of the valuations $\nu$ and $\overline{\nu}$
(the one defined by the component $E_{\overline{\delta}}$).
(The vertex $\delta$ may coincide with $\tau_g$.)

Let $\widehat{X}=X/{\overline{\cdot}}$,
$\widehat{D}=D/{\overline{\cdot}}$, and
$\widehat{\Gamma}=\Gamma/{\overline{\cdot}}$ be the quotients
of $X$, $D$, and $\Gamma$ by the complex conjugation.
The graph $\widehat{\Gamma}$ looks like on
Figure~\ref{fig2} and is the minimal resolution graph
of the valuation $\nu$ itself.
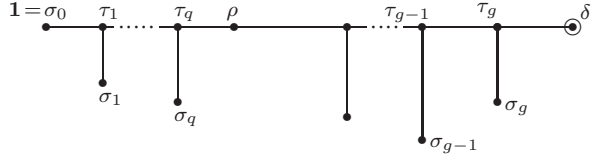
\begin{figure}[h]
$$
\unitlength=0.50mm
\begin{picture}(160.00,40.00)(0,5)
\thinlines


\put(-10,30){\line(1,0){18}}
\put(10,30){\circle*{0.5}}
\put(12,30){\circle*{0.5}}
\put(14,30){\circle*{0.5}}
\put(16,30){\circle*{0.5}}
\put(18,30){\circle*{0.5}}
\put(20,30){\line(1,0){55}}

\put(110,30){\circle*{2}}
\put(-10,30){\circle*{2}}
\put(5,30){\line(0,-1){15}}
\put(5,30){\circle*{2}}
\put(5,15){\circle*{2}}
\put(25,30){\line(0,-1){20}}
\put(25,30){\circle*{2}}
\put(25,10){\circle*{2}}

\put(40,30){\circle*{2}} 
\put(38,33){{\scriptsize$\rho$}} 
\put(70,30){\circle*{2}} 
\put(90,30){\circle*{2}} 
\put(110,30){\circle*{2}} 

\put(70,30){\line(0,-1){24}}
\put(90,30){\line(0,-1){30}}
\put(110,30){\line(0,-1){20}}
\put(70,6){\circle*{2}} 
\put(90,0){\circle*{2}} 
\put(110,10){\circle*{2}} 

\put(85,30){\line(1,0){45}} 

\put(77,30){\circle*{0.5}}
\put(79,30){\circle*{0.5}}
\put(81,30){\circle*{0.5}}
\put(83,30){\circle*{0.5}}

\put(-20,33){{\scriptsize ${\bf 1}\!=\!\sigma_0$}}
\put(4,33){{\scriptsize $\tau_1$}}
\put(24,33){{\scriptsize $\tau_q$}}
\put(80,33){{\scriptsize $\tau_{g-1}$}}
\put(105.5,34){{\scriptsize $\!\tau_g$}}
\put(132,32){{\scriptsize $\delta$}}

\put(4,10){{\scriptsize $\sigma_1$}}
\put(24,5){{\scriptsize $\sigma_q$}}
\put(92,-2){{\scriptsize $\sigma_{g-1}$}}
\put(112,8){{\scriptsize $\sigma_g$}}
\put(130,30){\circle*{2}}
\put(130,30){\circle{4}}
\end{picture}
$$
\caption{The quotient $\widehat{\Gamma}$ of the resolution graph $\Gamma$ by
the complex conjugation.}
\label{fig2}
\end{figure}
One has $\widehat{D}=
\bigcup_{\sigma\in\widehat{\Gamma}}
\widehat{E}_{\sigma}$, where $\widehat{E}_{\sigma}$ is the image of $E_{\sigma}$ in $\widehat{D}$.
($\widehat{E}_{\sigma}$ is either the disc (if
$\sigma=\overline{\sigma}$) or the complex projective line
(otherwise).)

Let the exceptional divisor $D$ be the union of its irreducible componets $E_\sigma$, $\sigma\in\Gamma$.
One has the natural involution of complex conjugation
on $\Gamma$. Let $(E_\sigma\circ E_{\sigma'})$ be the
intersection matrix of the components of the exceptional
divisor. (The self-intersection number $E_\sigma\circ E_\sigma$ is a negative integer and, for $\sigma'\ne\sigma$,
$E_\sigma\circ E_{\sigma'}$ is either $1$ (if $E_\sigma$ and  $E_{\sigma'}$ intersect) or $0$ otherwise.) Let
$(m_{\sigma\sigma'}):=-(E_\sigma\circ E_{\sigma'})^{-1}$.
(Each entry $m_{\sigma\sigma'}$ is a positive integer. It is equal to the intersection number $(C_{\sigma}\circ C_{\sigma'})$ between the curvettes $C_{\sigma}$ and $C_{\sigma'}$
at the components $E_{\sigma}$ and $E_{\sigma'}$ respectively.)
Let $m_{\sigma}:=m_{\sigma\delta}$.

Let $M_\sigma\in\Z_{>0}$ be defined in the following way.
If $E_{\overline{\sigma}}=E_{\sigma}$, one puts $M_\sigma=m_\sigma$; if
$E_{\overline{\sigma}}\ne E_\sigma$,
one puts $M_\sigma:=m_\sigma+m_{\overline\sigma}$.
In particular, $M_{\rho}=m_{\rho}$.

For $\sigma\in\Gamma$, let a curvette 
$(C_{\sigma},0)\subset(\C^2,0)$
at the component $E_{\sigma}$
be given by an equation $g_{\sigma}=0$,
$g_{\sigma}\in\OO_{\C^2,0}$.
(If $\sigma=\overline{\sigma}$, one may assume that the germs
$(C_{\sigma},0)$ and $g_{\sigma}$ are real ones.)
One can show that
$$
M_{\sigma}=\begin{cases}
            \nu(g_{\sigma}), & \text{if
            $\sigma=\overline{\sigma}$};\\
            \nu(g_{\sigma}\cdot\overline{g_{\sigma}}), & \text{if
            $\sigma\neq\overline{\sigma}$}.\\
           \end{cases}
$$

\section{The classical Poincar\'e series
of a divisorial valuation}\label{sec:classical}
Let $S=\{\nu(f): f\in\EE_{\R^2,0}\}$
be the semigroup of values of the valuation
$\nu$ on the algebra $\EE_{\R^2,0}$ and
let $P^S(t)=\sum\limits_{v\in S}t^v$.
Let $C$ be a curvette at the component
 $E_{\delta}$ (i.\,e., the image under the projection $\pi$ of a germ $\gamma$ of a
curve on $X$ transversal to the
 component $E_{\delta}$ at a smooth point
 of the exceptional divisor $D$).
It is known that the semigroup $S$ is equal to the semigroup
$S_C$ of the curve valuation $\nu_C$.
Then, one has
$$
P^S(t)=\frac{\prod_{i=1}^g(1-t^{M_{\tau_i}})}
 {\prod_{i=0}^g(1-t^{M_{\sigma_i}})}
$$
(see~\cite[Theorem 4]{BLMS-2024}).

One can see that,
 if the component $\delta$ is real, that is
 $\delta=\overline{\delta}$, the Poincar\'e
 series $P_{\nu}(t)$ coincides with the usual
 Poincar\'e series of the (divisorial) valuation
 $\nu$ (that is defined in the complex setting and computed in~\cite{Galindo-1995}):
 $$
 P_{\nu}(t)=\frac{1}{1-t^{m_{\delta}}}\cdot
 \frac{\prod_{i=1}^g(1-t^{m_{\tau_i}})}
 {\prod_{i=0}^g(1-t^{m_{\sigma_i}})}.
 $$
 (In this case $M_{\sigma}=m_{\sigma}$ for all $\sigma$.)
 Therefore from now on we assume that $\delta\ne\overline{\delta}$.
 In particular, $\delta$ is greater than $\rho$, i.\,e.,
 $\rho$ lies on the geodesic from $1$ to $\delta$.

Let $P_{\nu_C}(t)$ be the classical Poincar\'e
 series of the curve valuation $\nu_C$ defined by $C$.
  According to~\cite[Theorem 8]{BLMS-2024},
 one has
 $$
 P_{\nu_C}(t) = (1+t^{M_\rho})\cdot \frac{\prod_{i=1}^g(1-t^{M_{\tau_i}})}
 {\prod_{i=0}^g(1-t^{M_{\sigma_i}})}
{\color{red}=
(1+t^{M_\rho})\cdot P^S(t) \,.}
 $$
 Each coefficient $a_v$ of the series $P_{\nu_C}(t)=
 \sum_{v=0}^{\infty}a_vt^v$
 is equal to $0$, $1$ or $2$ (depending on whether
 $J(v)/J(v+1)$ is trivial, is the real line or is the complex line respectively).
The condition that $\delta>\rho$
 implies that $a_{M_\delta}=2$: see~\cite[Proposition~5, Lemmas~6, 7]{BLMS-2024}.

\begin{theorem}\label{theo:classic_divis}
 The classical Poincar\'e series of the divisorial
 valuation $\nu$ is equal to
 \begin{equation}\label{eqn:classic_divis}
  P_{\nu}(t)=\frac{1}{1-t^{M_{\delta}}}\cdot P_{\nu_C}(t)=
  \frac{1+t^{M_{\rho}}}{1-t^{M_{\delta}}}
  \cdot \frac{\prod_{i=1}^g(1-t^{M_{\tau_i}})}
 {\prod_{i=0}^g(1-t^{M_{\sigma_i}})}\,.
 \end{equation}
\end{theorem}

\begin{proof}
 Let $b_v$ be the coefficients of the classical
 Poincar\'e series of the valuation $\nu$:
 $P_{\nu}(t)= \sum_{v=0}^{\infty}b_vt^v$.
 For $v\in S$, let $k$ be the maximal non-negative integer
 such that $v=kM_{\delta}+r$ with $r\in S$.

 \begin{lemma}\label{lemma:b_v}
  One has $b_v=2k+a_r$.
 \end{lemma}

 \begin{proof}
 Making (if necessary) additional blow-ups at intersection
 points of the components of the total transform
 $\pi^{-1}(C)(=\pi^{-1}(0)\cup\gamma)$, one may assume that, for any  $f\in
\OO_{\C^2,0}$ with
$\nu(f)\le v$,
the strict transform of the curve germ
$\{f=0\}$ intersects the exceptional divisor
 $\pi^{-1}(0)$ only at smooth points of $\pi^{-1}(C)$.
 To compute the dimension $b_v$ of the space
 $J_{\R}(v)/J_{\R}(v+1)$, let us consider the sets of functions
 $f\in\EE_{\R^2,0}$ with fixed intersection numbers of the strict transforms of the curves $\{f=0\}$
 with all the components of the exceptional divisor $D$ and
 the (topological!) dimensions of their images in $J_{\R}(v)/J_{\R}(v+1)$. (There
are finitely many non-empty sets of this sort.) The statement will follow from the fact that all
 these dimensions are $\le 2k+a_r$ and (at least) one of them is equal to $2k+a_r$.

 If $s$ is the intersection number with the component
 $E_{\delta}$, one has $s\le k$. Let us consider the case $s<k$
 first. Each function $f$ from a set of the described type can be represented
 as the product $f'\cdot f^{''}$. where the strict transform of the curve
$\{f'=0\}$ does not intersect the component
 $E_{\delta}$ and the strict transform of the curve $\{f{''}=0\}$ intersects only
the component $E_{\delta}$ (with multiplicity $s$). (This representation is unique up to
 multiplication by a real function germ with a non-zero
 value at the origin.) The dimension of the image of the
 set of functions $f'$ in
 $J_{\R}(v-sM_{\delta})/J_{\R}(v-sM_{\delta}+1)$ is $\le 2$.
(This follows from the fact that the restriction of the ratio of the liftings
of two functions from this set to the component $E_{\delta}$ of the exceptional divisor
is a (non-zero) constant: real or complex.)
 The dimension of the image of the set of functions $f{''}$ in
 $J_{\R}(sM_{\delta})/J_{\R}(sM_{\delta}+1)$ is equal
 to $2s+1$. (Up to a constant real factor the elements in the image are in
bijective correspondence with the spaces of divisors
 of the intersection $\{f{''}=0\}\cap D$ whose dimension is $2s$.) The both sets
are invariant with respect to the multiplication by $\R^*$. Thus the dimension of the image of the set of functions
 $f$ is $\le 2s+2<2k+a_r$. For $s=k$,
 the dimension of the image of the
 set of functions $f'$ in
 $J_{\R}(v-k M_{\delta})/J_{\R}(v-k M_{\delta}+1)$ is $\le a_r$
 for each of the sets and is equal to $a_r$ for (at least) one of them.
 The dimension of the image of the set of functions $f{''}$ in
 $J_{\R}(k M_{\delta})/J_{\R}(k M_{\delta}+1)$ is equal
 to $2k+1$. Thus the dimension of the image of each of the sets of functions
 $f$ is $\le 2k+a_r$ and is equal to $2k+a_r$ for (at least) one of them. This proves the statement.
 \end{proof}

 Since $\delta>\rho$,
$M_\delta = m_\delta + m_{\overline{\delta}} = m_\delta + \ell m_\rho$ for some
integer
$\ell\ge 1$ and so  one has $a_{r+s M_{\delta}}=2$ for
$s>0$ (see \cite{BLMS-2024}).
Therefore
 $$
 \sum_{v=0}^{\infty}b_vt^v=
 \frac{1}{1-t^{M_{\delta}}}\cdot
 \sum_{v=0}^{\infty}a_vt^v.
 $$
 \end{proof}

\section{The real Poincar\'e series
of a divisorial valuation}\label{sec:real}

\begin{theorem}\label{theo:real_divisorial}
One has
\begin{align*}
\text{if $\delta=\b\delta$, then \ } & \
P_\nu^{\chi}(t) = P_\nu(t)\cdot\frac{1-t^{m_\delta}}{1+t^{m_\delta}}; \\
\text{if $\delta \neq \b\delta$, then \ } &\
P_\nu^{\chi}(t) = P_\nu(t)\cdot\frac{1-t^{m_\rho}}{1+t^{m_\rho}}.
\end{align*}
\end{theorem}

\begin{proof} The coefficients of the series
$P_\nu^{\chi}(t)$ are equal to the residues
of the corresponding coefficients $b_v$ of the
series $P_{\nu}(t)= \sum_{v=0}^{\infty}b_vt^v$ modulo $2$, i.\,e.,
the coefficient of $t^v$ in $P_\nu^{\chi}(t)$
is equal to $0$ if $b_v$ is even and is equal
to $1$ if $b_v$ is odd. The equations for
$P_{\nu}(t)$ are somewhat different for real
$\delta$ (i.\,e., $\delta= \b{\delta}$) and for $\delta$ not real
(i.\,e., $\delta\neq  \b{\delta}$).
Let us consider
these cases separately.

\noindent{\bf Case $\delta=\b{\delta}$.}
One can write
$$
P_\nu(t)\cdot\frac{1-t^{m_\delta}}{1+t^{m_\delta}} =
\frac{P^S(t)}{1-t^{m_\delta}}\cdot
\frac{1-t^{m_\delta}}{1+t^{m_\delta}} =
P^S(t)\cdot\frac{1-t^{m_\delta}}{1-t^{2 m_\delta}} =
Q(t)\cdot(1-t^{m_\delta})\,.
$$
The coefficients of the series $Q(t)=\sum_{v=0}^{\infty} q_v\cdot t^v$ can be
computed in the following way. If $v\notin S$, then $q_v=0$. Let $v\in S$ be such
that $v = k(2 m_\delta)+v'$, where $v'\in S$ and
$v'-2 m_\delta\notin S$. Then $q_v=k+1$.
Now,
$$
Q(t)\cdot(1-t^{m_\delta}) =
\sum_{v=0}^{\infty}(q_v-q_{v-m_\delta})t^v.
$$
Let $v'=\eps m_\delta+v''$, where $\eps$
is equal to $0$ or $1$, $v''\in S$, and
$v''-m_\delta\notin S$. Then one has the following possibilities:
\begin{itemize}
\item
If $\eps=1$, then $v-m_\delta= k(2m_\delta)+v''$ and therefore
$q_{v-m_\delta}=k+1$, $q_v-q_{v-m_{\delta}}=0$.
Moreover,
$v = k(2m_\delta)+ m_\delta + v'' = (2k+1)m_\delta+v''$ and thus
$b_v = 2k+2 \equiv 0 \mod 2$.
\item
If $\eps=0$, then $v = k(2m_\delta)+v''$,
$v-m_\delta= (k-1)(2m_\delta) + (m_\delta+\epsilon')$ and therefore
$q_{v-m_\delta}=k$ and $q_v-q_{v-m_{\delta}}=1$.
Moreover,
$v = k(2m_\delta)+v''$ and $v''-m_\delta\notin S$, thus
$b_v = 2k+1 \equiv 1 \mod 2$.
\end{itemize}

\medskip

\noindent{\bf Case $\delta\neq\b{\delta}$.}
One can write
$$
P_\nu(t)\cdot \frac{1-t^{m_\rho}}{1+t^{m_\rho}} =
\frac{P_{\nu_C}(t)}{1-t^{M_\delta}}\cdot \frac{1-t^{m_\rho}}{1+t^{m_\rho}} =
\frac{P^S(t)\cdot (1+t^{m_\rho}) }{1-t^{M_\delta}}\cdot \frac{1-t^{m_\rho}}{1+t^{m_\rho}} =
Q(t)\cdot(1-t^{m_\rho}) \;,
$$
where $Q(t)=P^S(t)/(1-t^{M_\delta})=\sum q_vt^v$.
The coefficients $q_v$ can be computed as above, i.\,e.,
write $v= k M_\delta+v'$ such that $v'-M_\delta\notin S$, then
$q_v=k+1$.
Since $P_\nu(t)=\sum b_vt^v=Q(t)\cdot (1+t^{m_\rho})$, one has
$b_v=q_v+q_{v-m_\rho}$ for all $v\in \Z$.

Now, $v-m_\rho = k M_{\delta} + (v'-m_\rho)$ and therefore
\begin{itemize}
\item
If $v'-m_\rho\in S$, then $q_{v-m_\rho}=q_v=k+1$ and thus
$b_v=2 q_v=2(k+1)\equiv 0\mod 2$.
\item
If $v'-m_\rho\notin S$, then
$v-m_\rho = (k-1)M_\delta + (M_\delta -m_\rho +v')$.
We know that $M_\delta = m_\delta + \ell m_\rho$ with
$\ell\in \Z, \ell\ge 1$. Thus
$v''=M_\delta-m_\rho+v'\in S$ and $v''-M_{\delta}\notin S$.
This implies that
$q_{v-m_\rho}=k=q_v-1$ and therefore
$b_v=2 q_v-1=2k+1\equiv 1\mod 2$.
\end{itemize}

On the other hand, $Q(t)(1-t^{m\rho}) = \sum (q_v-q_{v-m_\rho})\cdot t^v$ and
the computations above show also that
the coefficients are
$q_v-q_{v-m_\rho}=0$ if $v'-m_\rho\in S$ and
$q_v-q_{v-m_\rho}=1$ if $v'-m_\rho\notin S$.
This proves the statement of the Theorem.
\end{proof}

\begin{remark}
 The equations from Theorems~\ref{theo:classic_divis} and~\ref{theo:real_divisorial}
 imply that each of the series $P_{\nu}(t)$ and $P_\nu^{\chi}(t)$ determines the minimal {\bf real} resolution of the
 divisorial valuation $\nu$.
\end{remark}

Addresses:

A. Campillo and F. Delgado:
IMUVA (Instituto de Investigaci\'on en
Matem\'aticas), Universidad de Valladolid,
Paseo de Bel\'en, 7, 47011 Valladolid, Spain.
\newline E-mail: campillo\symbol{'100}agt.uva.es, fdelgado\symbol{'100}uva.es

S.M. Gusein-Zade:
Moscow State University, Faculty of Mathematics and Mechanics, Moscow Center for
Fundamental and Applied Mathematics, Moscow, Leninskie Gory 1, GSP-1, 119991, Russia.\\
\& National Research University ``Higher School of Economics'',
Usacheva street 6, Moscow, 119048, Russia.
\newline E-mail: sabir\symbol{'100}mccme.ru

\end{document}